\documentclass[11pt]{amsart}
\usepackage{amsfonts, amsmath, amssymb, amsthm, amscd, color, bookmark,graphicx}
\usepackage[latin1]{inputenc}
\usepackage[shortlabels]{enumitem}
\usepackage{tikz-cd}

\usepackage{hyperref}
\hypersetup{
  pdfborder={0 0 0},
  colorlinks   = true, 
  urlcolor     = blue, 
  linkcolor    = blue, 
  citecolor   = red 
}

\allowdisplaybreaks

\newtheorem{thm}{Theorem}[section]

\newtheorem{lemma}[thm]{Lemma}

\theoremstyle{definition}
\newtheorem{defn}[thm]{Definition}

\theoremstyle{remark}

\newtheorem{rem}[thm]{Remark}

\newcommand{\C}{\mathrm{C}}

\newcommand{\e}{\varepsilon}

\newcommand{\N }{\mathbb{N}}
\newcommand{\Z }{\mathbb{Z}}

\numberwithin{equation}{section}

\begin{document}

\vspace*{-1cm}
\title{Some examples of invariably generated groups}

\author{Ashot Minasyan}
\address{School of Mathematical Sciences,
University of Southampton, Highfield, Southampton, SO17 1BJ, United
Kingdom.}
\email{aminasyan@gmail.com}
%

\begin{abstract}
A group $G$ is invariably generated (IG) if there is a subset $S \subseteq G$ such that for every subset $S' \subseteq G$, obtained from $S$ by replacing each element with a conjugate, $S'$ generates $G$. Likewise, $G$ is finitely invariably generated (FIG) if, in addition, one can choose such a subset $S$ to be finite.

In this note we construct a FIG group $G$ with an index $2$ subgroup $N \lhd G$ such that $N$ is not IG. This shows that neither property IG nor FIG is stable under passing to subgroups of finite index, answering questions of Wiegold and Kantor-Lubotzky-Shalev. We also produce first examples of finitely generated IG groups that are not FIG, answering a question of Cox.
\end{abstract}


\maketitle

\section{Introduction}
A subset $S$ \emph{invariably generates} a group $G$ if for every function $f:S \to G$ the subset $\{s^{f(s)} \mid s \in S\}$ is a generating set of $G$.
We say that $G$ is \emph{invariably generated} (IG) if it contains an invariable generating subset (equivalently, if $G$ invariably generates itself).
Similarly, $G$ is \emph{finitely invariably generated} (FIG) if it has a finite invariable generating subset.

The term ``invariably generated'' was invented by Dixon \cite{Dix} in 1988, though the notion itself appeared in the literature earlier. In 1975 Wiegold \cite{Wie1} considered the class of groups $\mathcal X$ such that $G \in \mathcal X$ if and only if for every transitive action of $G$ on a set
$\Omega$, with $|\Omega| \ge 2$,  at least one element $g \in G$ acts on $\Omega$ without fixed points. The fact that the latter property holds for all finite groups was proved by Jordan \cite{Jor} in 1872, more that 100 years earlier. Jordan's theorem was revisited by Serre \cite{Ser} in 2003, who gave several applications to Number Theory and Topology.

In \cite{Wie1} Wiegold observed the following.

\begin{rem}\label{rem:IG-equiv_def} For a group $G$ the following statements are equivalent:
\begin{itemize}
  \item $G \in \mathcal X$;
  \item for each proper subgroup $H <G$, $\bigcup_{g \in G} H^g \neq G$;
  \item if $S$ is a subset of $G$ containing a representative from each conjugacy class, then $G=\langle S \rangle$; in other words, $G$ is IG.
\end{itemize}
\end{rem}

Given two subsets $A,B$ of a group $G$ we will say that $A$ is \emph{pointwise conjugate into} $B$ if $A \subseteq B^G$, i.e., each
$a \in A$ is conjugate to some $b\in B$. The above remark tells us that $G$ is IG if and only if $G$ is not pointwise conjugate into a proper subgroup.
Wiegold proved that the class $\mathcal X$ of IG groups is closed under extensions and restricted direct products \cite{Wie1}. Since finite groups (by
Jordan's theorem \cite{Jor}) and abelian groups are in $\mathcal X$, it follows that all virtually solvable groups are IG. On the other hand, the easiest examples of non-IG groups are non-abelian free groups \cite{Wie1}.

In 2014 Kantor, Lubotzky and Shalev \cite{K-L-S} studied invariable generation for infinite linear groups and defined the property FIG.
Similarly to Remark~\ref{rem:IG-equiv_def}, a group is FIG if and only if there is a finite subset $S \subseteq G$ which is not pointwise conjugate into any proper subgroup of $G$.  One of the main results from \cite{K-L-S} states that a finitely generated linear group is FIG if and only if it is virtually solvable.
They also showed that the class of FIG groups is closed under extensions and contains all finitely generated abelian-by-polycyclic groups.
In fact, the argument from  \cite[Lemma 2.8]{K-L-S} (see also \cite[Theorem 2.1.(i)]{Wie1}) proves the following.

\begin{lemma}\label{lem:IG_for_ext}
Suppose that $G$ is a group, $N \leqslant K \leqslant G$ and $N \lhd G$. If $S$ is an invariable generating subset of $K$ and $T$
is an invariable generating subset of $G/N$ then $S \cup T_1$ is an invariable generating subset of $G$, where $T_1 \subseteq G$ is any preimage of $T$ in $G$.
\end{lemma}

This lemma easily yields that each of the  properties IG and FIG is inherited by finite index overgroups. Wiegold \cite[p. 573]{Wie2} and Kantor, Lubotzky, Shalev \cite[Open Problems 1,2]{K-L-S} asked whether the same can be said for finite index subgroups. Our first result gives a negative answer to these two questions.

\begin{thm}\label{thm:fin_ind} There exists a torsion-free group $G$ with a subgroup $N \lhd G$, of index $2$, such that $G$ is invariably generated by two elements but $N$ is not invariably generated. More precisely, there are two elements $a, b \in G$ such that all of the following hold.
\begin{itemize}
  \item[(i)] $\{a,b\}$ invariably generates $G$;
  \item[(ii)] the subgroup $N=\langle a^2,b,aba^{-1} \rangle$ has index $2$ in $G$;
  \item[(iii)] $H=\langle a^2,b \rangle$ is free of rank $2$, freely generated by $\{a^2,b\}$, and $|N:H|=\infty$;
  \item[(iv)] $N= \bigcup_{f \in N} H^f$.
 \end{itemize}
\end{thm}

Note that the first example of a non-IG subgroup of an IG group was given by Wiegold in \cite{Wie2}.

More examples of IG, FIG and non-IG infinite groups have been recently obtained in \cite{Gel,G-G-J,B-Fuj,Cox}.
Property FIG implies IG, by definition; in fact, FIG is strictly stronger as any non-finitely generated abelian group is IG but not FIG. However, prior to this work none of the known examples answered the following basic question, asked by Cox in \cite{Cox}: is property IG equivalent to FIG for finitely generated groups? Our second theorem shows that this is not the case.

\begin{thm}\label{thm:IG_not_FIG}
There exists a  finitely generated torsion-free group $G$ which is IG but not FIG. More precisely, $G$ is generated by two elements and
there is an infinite family of
subgroups $H_j < G$, $j=0,1,2,\dots,$ such that all of the following conditions are satisfied.
\begin{itemize}
  \item[(i)] For all $j \in \N\cup \{0\}$, $H_j$ is a malnormal finitely generated  free subgroup of $G$;
  \item[(ii)] for each $i \in \N$ there exists  $r \in H_i\setminus\{1\}$ such that $\langle r \rangle \cap gH_jg^{-1}=\{1\}$, for any
        $j<i$ and all $g \in G$;
  \item[(iii)] every finite subset of $G$ is pointwise conjugate into $H_j$, for some $j \in \N \cup \{0\}$;
  \item[(iv)] the intersection $hH_j h^{-1} \cap H_i$ is cyclic, for all $i \neq j$ and all $h \in G$;
  \item[(v)] if $M $ is a proper subgroup of $G$,  then there exists $j \in \N \cup \{0\}$ and $v \in G$
such that $M \subseteq v H_j v^{-1}$.
 \end{itemize}
\end{thm}

Recall that a subgroup $H$ of a group $G$ is called \emph{malnormal} if $gHg^{-1} \cap H=\{1\}$ for all $g \in G \setminus H$. The fact that any group $G$ satisfying properties
(i)--(v) from Theorem~\ref{thm:IG_not_FIG} is not FIG follows from  (iii). Condition (ii) shows that $G \neq \bigcup_{g \in G} H_j^g$ for any given $j \in \N \cup\{0\}$, and since each proper subgroup of $G$
is conjugate into some $H_j$, by (v), we can use Remark~\ref{rem:IG-equiv_def} to conclude that $G$ is IG.

Thus even for finitely generated groups property FIG is more restrictive than property IG.
In view of Theorems~\ref{thm:fin_ind} and \ref{thm:IG_not_FIG}, \cite[Corollary A]{Cox} implies that every finitely generated group can be embedded in a FIG group that has a non-IG subgroup of finite index or in a finitely generated IG non-FIG group. In particular, there is a continuum of such groups.

After this paper was written, the author learned that Goffer and Lazarovich obtained similar results to Theorems~\ref{thm:fin_ind},\ref{thm:IG_not_FIG} independently in \cite{G-L}.
Their work also answers the questions of Wiegold, Kantor-Lubotzky-Shalev and Cox mentioned above. Both our paper and \cite{G-L} are based on
the small cancellation theory over hyperbolic groups developed by Ol'shanskii \cite{Olsh}. However, \cite{G-L} works with this theory directly, while our approach
uses the author's previous results from \cite{Min-G-sbgps}.

\subsection{Ideas of proofs} Let us briefly outline the constructions of the groups from Theorems~\ref{thm:fin_ind} and \ref{thm:IG_not_FIG}. In both of them the group $G$ is obtained as a direct limit of torsion-free hyperbolic groups, using a ``small cancellation quotient theorem''. In the original form such a theorem was proposed by Gromov \cite[5.5]{Gromov}, and then proved by Ol'shanskii in \cite{Olsh}. The statement has later been extended to relatively hyperbolic groups by Osin \cite{Osin} and then to acylindrically hyperbolic groups by Hull \cite{Hull}. The version that we will use here was obtained by the author  in \cite{Min-G-sbgps, Min-thesis} and allows to preserve certain quasiconvex subgroups in the small cancellation quotient: see Theorem~\ref{thm:pres_qc} below.
This version has an advantage over the versions from \cite{Osin,Hull} in
not having to require that the preserved collection of quasiconvex subgroups be malnormal in the ambient hyperbolic group, which is important for the constructions that we present here.

To prove Theorem~\ref{thm:fin_ind} we start with a group $G_0$, freely generated by $\{a_0,b_0\}$, let
$H_0=\langle a_0^2,b_0\rangle$ and $N_0=\langle a_0^2,b_0,b_0^{a_0}\rangle \lhd G_0$. We enumerate all elements of $G_0$: $f_0=1$, $f_1,f_2,\dots$, and construct non-elementary torsion-free hyperbolic groups $G_1,G_2,\dots$ by induction, so that for each  $G_{n+1}$ is a quotient of $G_{n}$, $n \ge 0$.
To obtain $G_{n+1}$ from $G_n$, together with an  epimorphism $\phi_{n+1}:G_n \to G_{n+1}$,
we first embed $G_n$ into an HNN-extension $L_n$, in which the image of $f_{n+1}$ is conjugate to an element of the image of $H_0 \cup \langle a_0\rangle$. We then use Theorem~\ref{thm:pres_qc} to produce a torsion-free hyperbolic group $G_{n+1}$, and an epimorphism $\eta_{n+1}:L_n \to G_{n+1}$ such that $\eta_{n+1}$
is injective on the image of $H_0$ in $L_n$ and $G_{n+1}$ is generated by the images of $a_0$ and $b_0^{f_{n+1}}$. In particular, $\eta_{n+1}(G_n)=G_{n+1}$ and we set
$\phi_{n+1}:G_n \to G_{n+1}$ to be the restriction of $\eta_{n+1}$ to $G_n$. We define the epimorphism $\psi_{n+1}:G_0 \to G_{n+1}$ by $\psi_{n+1}=\phi_{n+1} \circ \dots\circ \phi_1$, and set
$a_{n+1}=\psi_{n+1}(a_0)$, $b_{n+1}=\psi_{n+1}(b_0)$, $H_{n+1}=\psi_{n+1}(H_0)$ and $N_{n+1}=\psi_{n+1}(N_0)$.
By carefully controlling  the kernel of $\eta_{n+1}$ we ensure that
$N_{n+1}$ still has index $2$ in $G_{n+1}$ and $\psi_{n+1}(f_{n+1})$ is conjugate into $H_{n+1} \cup \langle a_{n+1}\rangle$ via an element of $N_{n+1}$ in $G_{n+1}$.
Finally we define $G$ as the direct limit of the sequence $(G_n)_{n=0}^\infty$, let $a,b \in G$ and $H,N \leqslant G$ be the natural images of $a_0, b_0 \in G_0$ and $H_0,N_0 \leqslant G_0$ respectively.
By construction, every $G_i$ is generated by $a_i$ and $b_i^{\psi_i(f_i)}$, whence $G=\langle a,b^f\rangle$ for all $f \in G$, so that $\{a,b\}$ is an invariable generating set of $G$. On the other hand, the condition that $\psi_i(f_i) \in H_i^{N_i} \cup \langle a_i \rangle ^{N_i}$ in $G_i$, $i \in \N$, is sufficient for showing that $N =\bigcup_{f \in N} H^f$.
The actual argument is  somewhat technical, since various conditions need to be preserved at each step in order for the inductive argument to work and for Theorem~\ref{thm:pres_qc} to be applicable.

The proof of  Theorem~\ref{thm:IG_not_FIG} is more involved because the difference between properties IG and FIG is quite subtle. We start with a free group $G_0$ of rank $3$, set $H_{00}=\{1\}$, and enumerate all of the finite subsets $S_0=\emptyset$, $S_1,S_2, \dots$, and all of the
finitely generated subgroups $Y_0=\{1\}$, $Y_1,Y_2,\dots$ in $G_0$.
The torsion-free hyperbolic group $G_{n+1}$, together with a collection of malnormal quasiconvex subgroups $H_{n+1,j} \leqslant G_{n+1}$, $j=0,\dots,n+1$,
are constructed by induction on $n$. We define $L_n$ as an HNN-extension of the free product $G_n*X$, where $X$ is a free group of finite rank
and the image of $S_{n+1}$ is pointwise conjugate into $X$ in $L_n$. Thereafter $G_{n+1}$  is produced as a small cancellation quotient of $L_n$
in such a way that the natural homomorphism $\eta_{n+1}:L_n \to G_{n+1}$ is surjective on $G_n$ and on the image of $Y_{n+1}$ in $G_n$, unless this image is
cyclic or is contained in a conjugate of $H_{nj}$ in $L_n$, for some $j=0,\dots,n$. We let $H_{n+1,j} \leqslant G_{n+1}$ be the image of $H_{nj}$, for $j \leq n$, and $H_{n+1,n+1}=\eta_{n+1}(X)$. As before, $G$ is defined as the direct limit of the sequence $(G_n)_{n=0}^\infty$, and for each $j \ge 0$, $H_j$
is defined as the image of $H_{jj}$ under the natural epimorphism $G_j \to G$. It is straightforward from the construction that every proper finitely
generated subgroup of $G$ (being the image of some $Y_j$), will be contained in a conjugate of $H_j$, for some $j \ge 0$. The much stronger condition (v)
is proved by combining the malnormality of each subgroup $H_{jj} \leqslant G_j$ with property (iv) from the claim of the theorem.

The constructions we use to prove Theorems~\ref{thm:fin_ind},\ref{thm:IG_not_FIG} are fairly flexible and one can produce examples with additional properties: see Remarks~\ref{rem:cyc_centr-1},\ref{rem:add_props_Thm1} and \ref{rem:cyc_centr-2}.

\subsection*{Acknowledgements} The author would like to thank Charles Cox whose talk on the paper \cite{Cox} motivated this work, and who gave useful comments on a preliminary version of
this article. The author also thanks the anonymous referee for many useful suggestions which led to improvements of the exposition.

\section{Preliminaries}
\subsection{Notation and terminology}
In this paper we will denote by $\Z$ the set of all integers, by $\N=\{1,2,\dots\}$ the set of all natural numbers, and by $2\N=\{2,4,\dots\}$ the set of all even natural numbers.

Let $G$ be a group, $x,y \in G$ be any elements and  $A,B \subseteq G$ be any subsets. We will use the following notation:
\[x^y=yxy^{-1},~~ A^x=xAx^{-1},~~ A^B=\{bab^{-1} \mid a \in A, b \in B \}.\]

We will say that $x$ \emph{is commensurable to} $y$ in $G$, $x \stackrel{G}{\approx} y$, if some non-zero power of $x$ is conjugate to some non-zero power of $y$ in $G$. Note that $\stackrel{G}{\approx}$ is an equivalence relation on $G$.

As usual, $\C_G(x)=\{g \in G \mid gx=xg\}$ will denote the \emph{centralizer} of $x$ in $G$.

\begin{defn}\label{df:pres_conj}
Let $\eta:L \to G$ be a homomorphism between groups $L$ and $G$, and let $Q \subseteq L$ be any subset. We will say that \emph{$\eta$ preserves conjugacy on $Q$}
if for all $x,y \in Q$, $\eta(y) \in \eta(x)^G$ in $G$ implies that $y \in x^L$ in $L$. We will also say that \emph{$\eta$ preserves centralizers on $Q$} if
$\C_G(\eta(x))=\eta(\C_L(x))$ for all $x \in Q$.
\end{defn}

Finally, recall that a group is called \emph{elementary} if it has a cyclic subgroup of finite index. Throughout the paper we will often use the basic fact that
a torsion-free elementary group is necessarily cyclic.

\subsection{One subgroup of a free group}
The following elementary facts about a specific subgroup of a free group of rank $2$ will be used in the proof of Theorem~\ref{thm:fin_ind}.

\begin{lemma}\label{lem:H-props} Let $G$ be a free group of rank $2$ with free generating set $\{a,b\}$,  let $H=\langle a^2,b\rangle \leqslant G$ and
$N=\langle a^2,b,b^a \rangle\leqslant G$.
Then
\begin{itemize}
  \item[(i)] $N \lhd G$ and $|G:N|=2$, in particular, $|N:H|=\infty$;
  \item[(ii)] $H \cap H^{a}=\langle a^2 \rangle$; 
  \item[(iii)] if $a^l  \in h^G$, for some $h \in H$ and some  $l \in \Z$, then $a^l \in h^H$;
  \item[(iv)] $\C_G(a^l)=\langle a \rangle$ for any $l \in \Z\setminus\{0\}$, and $\C_G(u) \subseteq H$ if $u \in H \setminus \langle a^2 \rangle ^H$.
\end{itemize}
\end{lemma}

\begin{proof} Claim (i) follows from the observation that $N$ is the kernel of the homomorphism from $G$ onto $\Z/2\Z$, sending $a$ to $\overline{1}$ and $b$ to $\overline{0}$,
which also implies that $\{a^2,b,b^a\}$ is a free generating set of $N$.
Therefore $H \cap H^a =\langle a^2 \rangle$, as $H^a=\langle a^2,b^a \rangle$, so (ii) holds.

Now, suppose that $ h =xa^lx^{-1} \in H$ for some $x \in G$ and $l \in \Z$. Since $a \notin N$ and $H \subset N$ we see that $l$ must be even, so $a^l \in H$. As $|G:N|=2$, there exist
$k \in \{0,1\}$ and $t \in N$ such that $x=t a^k$, so $h=ta^l t^{-1}$ in $N$. Observe that $H$ is malnormal in $N$, being a free factor, hence either
$l=0$ and $h=1$ or $t \in H$. This proves claim (iii).

The fact that $\C_G(a^l)=\langle a \rangle$, for any $l \in \Z\setminus\{0\}$, is obvious. Now suppose that $u \in H$ is an element such that
$\C_G(u) \not\subseteq H$. Since $H$ is malnormal in $N$, there must exist $v \in \C_G(u)\setminus N$. Thus $v=fa$ for some $f \in N$, so
$u=u^v =fu^af^{-1}$, and $w=f^{-1}uf=u^a \in H^a$ is conjugate to $u \in H$ in $N$.

Recall that $N$ is freely generated by  $S=\{a^2,b,b^a\}$, and  $H=\langle a^2,b\rangle $,
$H^a= \langle a^2,b^a\rangle$ are freely generated by subsets of $S$.
Therefore we can write $u=hu_1h^{-1}$ and $w=h' w_1 h'^{-1}$, where $u_1,h \in H$, $w_1,h' \in H^a$, and the elements $u_1$, $w_1$ are cyclically reduced over $S$. The conjugacy criterion for cyclically reduced elements in the free group $N$ (see \cite[Proposition~I.2.14]{L-S})
implies that $u_1=a^{n}=w_1$, for some $n \in 2\N$.
Hence $u=h a^{n} h^{-1} \in \langle a^2 \rangle^H $, and (iv) holds.
\end{proof}

\subsection{Hyperbolic groups and quasiconvex subsets}
Let $G$ be a hyperbolic group in the sense of Gromov \cite{Gromov}. This means that $G$ is generated by a finite set $S$, and the Cayley graph $\Gamma(G,S)$ is $\delta$-hyperbolic, for some $\delta \ge 0$ (see \cite[Part III.H]{Bri-Hae} for a detailed exposition). A subset $Q$ of $G$ is said to be \emph{quasiconvex} if there exists $\e \ge 0$ such that any geodesic joining two elements of $Q$ in $\Gamma(G,S)$ is contained in the $\e$-neighborhood of $Q$.

We outline basic properties of quasiconvex subsets in the following remark.

\begin{rem}\label{rem:props-qc} Let $G$ be a hyperbolic group and let $Q \subseteq G$. Then
\begin{enumerate}
  \item the quasiconvexity of $Q$ is independent of the choice of a finite generating set $S$ for $G$;
  \item any elementary subgroup of $G$ is quasiconvex;
  \item if $G$ is a free group of finite rank then any finitely generated subgroup of $G$ is quasiconvex;
  \item finite unions and products of quasiconvex subsets of $G$ are also quasiconvex;
  \item if $K$ is a quasiconvex subgroup of $G$ then $K$ is itself hyperbolic and any quasiconvex subset of $K$ is also quasiconvex in $G$.
\end{enumerate}
\end{rem}

Properties (1)--(3) are well-known (cf. \cite[Chapter 3]{Mih} and \cite{Short}); see \cite[Lemma~2.1 and Proposition~0.1]{Min-pap1} for property (4).
Finally, property (5) follows from \cite[3.3.8]{Mih} and \cite[Remark~6]{Min-G-sbgps}.

\begin{lemma}\label{lem:elem}
If $G$ is  a torsion-free hyperbolic group and $c \in G$ is a non-trivial element then
$\C_G(c)$ is a malnormal infinite cyclic subgroup of $G$, generated by some $d \in G$. In particular,
$\C_G(d^m)=\langle d \rangle$, for any $m \in \Z\setminus\{0\}$.
\end{lemma}

\begin{proof} This is an immediate consequence of \cite[Lemma~1.16]{Olsh} and the fact that torsion-free elementary groups are cyclic.
\end{proof}

\begin{lemma}\label{lem:comm->conj_into_centr} Let $G$ be a torsion-free hyperbolic group and let $x,y, z \in G$ be elements
satisfying $y^l=zx^kz^{-1}$, for some $k,l \in \mathbb{Z}\setminus\{0\}$ . Then $y \in \C_G(x)^z$.
\end{lemma}

\begin{proof} By Lemma~\ref{lem:elem}, $\C_G(x)^z=\C_G(x^z)$ is malnormal in $G$, so, since $G$ is torsion-free, $y^l \in \C_G(x)^z$ implies that
$y \in \C_G(x)^z$.
\end{proof}

Following \cite{Min-G-sbgps} we say that a subset $Q$ of a group $G$ is \emph{small relative to a subgroup} $F \leqslant G$ if $F \not\subseteq P_1 Q^{-1}Q P_2$ for any finite subsets $P_1,P_2$ of $G$.

The next result was obtained by the author in \cite{Min-G-sbgps}, as part of his PhD thesis \cite{Min-thesis}, and generalized an earlier theorem of Ol'shanskii~\cite{Olsh}.

\begin{thm}\label{thm:pres_qc} Suppose that $L$ is a torsion-free hyperbolic group, $F \leqslant L$ is a non-elementary subgroup, $Q \subseteq L$ is a quasiconvex subset and
$U=\{u_1,\dots,u_n\}$ is any finite collection of elements of $L$. If $Q$ is small relative to $F$ in $L$ then there exist elements $w_1,\dots,w_n \in F$ satisfying the following. Let
$K \lhd L$ be the normal closure of the elements $u_iw_i$, $i=1,\dots,n$, let $G=L/K$ and let $\eta:L \to G$  be the natural homomorphism. Then
\begin{enumerate}[(a)]
  \item $G$ is a non-elementary torsion-free hyperbolic group;
  \item the restriction of $\eta$ to $Q$ is injective and $\eta(R)$ is quasiconvex in $G$ for every quasiconvex subset $R$ of $L$ with $R \subseteq Q$;
  \item $\eta$ preserves centralizers on $Q$ (in the sense of Definition~\ref{df:pres_conj});
 \item $\eta$ preserves conjugacy on $Q$ (in the sense of Definition~\ref{df:pres_conj}).
 \end{enumerate}
\end{thm}

\begin{proof} The statement is essentially a special case of \cite[Theorem~1]{Min-G-sbgps} (indeed, any non-elemen\-tary subgroup of a torsion-free hyperbolic group is a
\emph{$G$-subgroup} by \cite[Theorem 1]{Olsh}).
The main difference is that here we do not require the homomorphism $\phi$ to be surjective on $F$, but we simply ask for
the image of $F$ to contain the images of the given elements $u_i$, $i=1,\dots,n$. This allows us to specify that $\ker \phi=K$ is the normal closure of elements of the form $u_iw_i$, for some
$w_i \in F$, $i=1,\dots,n$. The latter can be seen from the explicit form of the extra relators imposed on $L$ to obtain $G$ in the proof of \cite[Theorem~1]{Min-G-sbgps}: see equation (21) and Section~7 in \cite{Min-G-sbgps}.
\end{proof}

In the case when $Q$ is a finite union of cosets of quasiconvex subgroups the condition of being small relative to $F$ was characterized in \cite{Min-G-sbgps} as follows.

\begin{thm}[{\cite[Theorem~3]{Min-G-sbgps}}]\label{thm:small-crit} Let $L$ be a hyperbolic group and let $F \leqslant L$ be any subgroup. Suppose that $H_1,\dots,H_k $ are quasiconvex subgroups of $L$ such that
$|F:(F \cap gH_ig^{-1})|=\infty$ for all $g \in L$ and all $i=1,\dots,k$. Then  the quasiconvex subset $Q=\bigcup_{i=1}^k H_i \subseteq L$ is small relative to $F$ in $L$.
\end{thm}

The next lemma shows that the map $\eta$ from Theorem~\ref{thm:pres_qc} preserves malnormality of any subgroup contained in $Q$.

\begin{lemma}\label{lem:malnorm_pres} Suppose that $\eta: L \to G$ is a homomorphism between groups $L$ and $G$, and  $H$ is a malnormal subgroup of $L$.
If $\eta$ preserves conjugacy and centralizers on $H$ then $\eta(H)$ is malnormal in $G$.
\end{lemma}

\begin{proof} Assume that $d=c^g$, for some $c,d \in \eta(H)\setminus\{1\}$ and some $g \in G$. Choose $a,b \in H\setminus\{1\}$ such that
$\eta(a)=c$ and $\eta(b)=d$. Since $\eta$ preserves conjugacy on $H$, $b=a^h$, for some $h \in L$, which implies that $h \in H$ as $H$ is malnormal in $L$.

Now, $c^g=d=c^{\eta(h)}$ in $G$, so $\eta(h)^{-1}g \in \C_G(c)=\eta(\C_L(a))$, since $\eta$ preserves centralizers on $H$. Moreover,
$\C_L(a) \subseteq H$, as $H$ is malnormal in $L$, so $g \in \eta(h)\eta(\C_L(a)) \subseteq \eta(H)$. Therefore $\eta(H)$ is malnormal in $G$.
\end{proof}

The following statement is essentially a corollary of one of the author's results from \cite{Min-pap2}.
\begin{lemma}\label{lem:cyc_intersec_preserved} Let $H_1$, $H_2$ be  subgroups of a hyperbolic group $L$, with $H_1$ quasiconvex,
 let $G$ be any group and let $\eta:L \to G$ be a homomorphism. Suppose that the following two conditions are satisfied:
\begin{itemize}
  \item[(i)] $\eta$ preserves conjugacy on $H_1 \cup H_2$;
  \item[(ii)] the intersection $H_1^g \cap H_2$ is elementary, for every $g \in L$.
\end{itemize}
Then for all $h \in G$,  the intersection $\eta(H_1)^{h} \cap \eta(H_2)$ is elementary in $G$.
\end{lemma}

\begin{proof} Choose any $h \in G$ and denote $A=\eta(H_1)^{h} \cap \eta(H_2) \leqslant G$. Thus $A \leqslant \eta(H_2)$ is conjugate into $\eta(H_1)$ in $G$.
Let $B \leqslant H_2$ be a preimage of $A$ under $\eta$. Then $B$ is pointwise conjugate into $H_1$ by (i), and, as $L$ is hyperbolic and
$H_1$ is quasiconvex, we can apply \cite[Proposition 1]{Min-pap2} to conclude that $|B:(B \cap H_1^g)|<\infty$ for some $g \in L$. But
$B \cap H_1^g \subseteq H_2 \cap H_1^g$ is elementary by condition (ii), hence $B$ is elementary, and so is $A=\eta(B)$.
\end{proof}

\subsection{HNN-extensions}
We will need the following statements about HNN-extensions.

\begin{lemma} \label{lem:HNN-hyp} Let $G$ be a hyperbolic group, let $X=\langle x \rangle$, $Y=\langle y \rangle$ be infinite cyclic subgroups of $G$.
Suppose that either $X$ or $Y$ is malnormal in $G$ and $x \stackrel{G}{\not\approx} y$. Then the HNN-extension
\begin{equation*}
L=G*_{X^t=Y}=\langle G,t \,\|\, txt^{-1}=y\rangle
\end{equation*}
is hyperbolic, and for any quasiconvex subgroup $H \leqslant G$, $H$ is quasiconvex in $L$.
\end{lemma}

\begin{proof} The hyperbolicity of $L$ under the above assumptions was first proved by Bestvina and Feighn \cite[Corollary~2.3]{B-F-corr}. Since cyclic subgroups
in hyperbolic groups are always quasiconvex (see Remark~\ref{rem:props-qc}), $G$ will be quasiconvex in $L$ by a result of
Kharlampovich and Myasnikov \cite[Theorem~4]{K-M}. Claim (5) of Remark~\ref{rem:props-qc} now shows that $H$ is quasiconvex in $L$.
\end{proof}

\begin{lemma} \label{lem:HNN-conj} Suppose that $G$ is a group and $X,Y\leqslant G$ are isomorphic subgroups, with an isomorphism $\tau:X \to Y$,
and $C \leqslant G$ is any subgroup. Let $L$ be the HNN-extension
\begin{equation*}
L=G*_{X^t=Y}=\langle G,t \,\|\, txt^{-1}=\tau(x),~x \in X\rangle.
\end{equation*}
\begin{itemize}
  \item[(a)] If $C$ and $Y$ are malnormal in $G$ and $C \cap Y^G=\{1\}$ then $C$ is malnormal in $L$.
  \item[(b)] If $u,v \in G$ are elements such that $v \notin u^G$ and $u \notin X^G \cup Y^G$ then $v \notin u^L$ in $L$.
\end{itemize}
\end{lemma}

\begin{proof} See \cite[Lemma~10.1]{Min-G-sbgps} for (a) and \cite[Lemma~3.4]{M-Z} for (b).
\end{proof}

\section{Invariable generation may not pass to finite index subgroups}

\begin{proof}[Proof of Theorem~\ref{thm:fin_ind}] The desired group $G$ will be constructed as a direct limit of torsion-free hyperbolic groups. Let $G_0$ be a free group freely generated by two elements $a_0,b_0 \in G_0$. Let $N_0=\langle a_0^2,b_0,a_0b_0a_0^{-1}\rangle$, so that $|G_0:N_0|=2$,
and let $H_0=\langle a_0^2,b_0\rangle$ and $Q_0=H_0 \cup H_0^{a_0}$. Then $\{a_0^2,b_0\}$ is a free generating set for $H_0$, and $H_0$, $Q_0$ are quasiconvex in $G_0$ by part (3) of Remark~\ref{rem:props-qc}. Let $\{f_0=1,f_1,f_2,\dots,\}$ be an enumeration of all elements of $G_0$.

For each $i \in \N$ we will construct a group $G_i$ and an epimorphism $\phi_i:G_{i-1} \to G_i$, so that for
$\psi_i=\phi_i \circ \dots \circ \phi_1 : G_0 \to G_i$ ,  $a_i=\psi_i(a_0) \in G_i$,
$b_i=\psi_i(b_0) \in G_i$, $g_i=\psi_i(f_i) \in G_i$, $N_i=\psi_i(N_0)=\langle a_i^2,b_i,a_ib_ia_i^{-1}\rangle \lhd G_i$, $H_i=\psi_i(H_0)=\langle a_i^2,b_i \rangle \leqslant G_i$ and $Q_i=\psi_i(Q_0)=H_i \cup H_i^{a_i}$ the following conditions hold:
%
%
%
\begin{flalign}
&G_i \text{ is a torsion-free hyperbolic group generated by }a_i,b_i;\label{eq:1}&&\\
&|G_i:N_i|=2, \text{ so that } \langle a_i\rangle \cap N_i=\langle a_i^2 \rangle;\label{eq:2}&&\\
&\psi_i \text{ is injective on } Q_0, ~H_i\text{ and } Q_i \text{ are quasiconvex in } G_i; \label{eq:3}&&\\
& \C_{G_i}(a_i^l)= \langle a_i \rangle, \text{ for all } l \in \Z\setminus\{0\},
\text{ and } \C_{G_i}(u) \subseteq H_i \text{ if } u \in H_i \setminus\langle a_i^2 \rangle ^{H_i}; \label{eq:4}&&\\
& \text{if } a_i^l \in h^{G_i}, \text{ for some } l \in 2\N \text{ and }  h \in H_i, \text{ then } a_i^l\in h^{H_i}; \label{eq:5}&&\\
& g_i\in H_i^{N_i} \cup \langle a_i \rangle^{N_i}; \label{eq:6}&&\\
& G_i=\langle a_i,b_i^{g_i} \rangle. \label{eq:7}&&
\end{flalign}

Setting $\psi_0:G_0 \to G_0$ to be the identity map, we see that for $i=0$ conditions \eqref{eq:1}--\eqref{eq:7} are all satisfied (either by construction or by Lemma~\ref{lem:H-props}). We now
proceed by induction, and thus assume that for some $n \in \mathbb{N}\cup \{0\}$, groups $G_0, \dots,G_n$, enjoying properties \eqref{eq:1}--\eqref{eq:7}, have already been constructed.
We are going to construct a group $G_{n+1}$ and an epimorphism $\phi_{n+1}:G_n \to G_{n+1}$ so that \eqref{eq:1}--\eqref{eq:7} hold for $i=n+1$.

Denote $c=\psi_n(f_{n+1}) \in G_n$. We shall first define an intermediate torsion-free hyperbolic group $L_n$, containing $G_n$, and an
index two normal subgroup
$M_n \lhd L_n$ such that $M_n \cap G_n =N_n$ and $c \in (H_n \cup \langle a_n \rangle)^t$, for some $t \in M_n$. If $c=1$ then set $L_n=G_n$, $M_n=N_n$ and $t=1$. So, assume that $c \neq 1$ and  let
$d \in G_n$ be an infinite order element such that $\C_{G_n}(c)=\langle d \rangle$ (cf. Lemma~\ref{lem:elem}).

If $d \in H_n^{N_n}$ in $G_n$ then set $L_n=G_n$, $M_n=N_n$ and let $t \in M_n$ be any element such that $d,c \in H_n^t$.
Similarly, if $d \stackrel{G_n}{\approx} a_n$, then $d,c \in \C_{G_n}(a_n)^z$, for some $z \in G_n$, by Lemma~\ref{lem:comm->conj_into_centr}.
So, recalling \eqref{eq:4}, we get
 $d,c \in \langle a_n \rangle^{t} $ in $G_n$, where $t \in N_n$ is the element such that $z=t a_n^\e$, for some $\e \in \{0,1\}$.
  Thus we can again take $L_n=G_n$, $M_n=N_n$ and $t$ as above.

So, further we can assume that $d$ is not commensurable to $a_n$ in $G_n$ and $d \notin H_n^{N_n}$.

\noindent\emph{Case 1:}  $d \notin N_n$. We claim that $d$ cannot be commensurable to any element from $H_n$, so
\begin{equation}\label{eq:Case_1-non-comm}
\langle d \rangle^{G_n} \cap H_n=\{1\}.
\end{equation}
Indeed, if $d \stackrel{G_n}{\approx} u \in H_n$,
then $d$ would be conjugate into $\C_{G_n}(u)$ by Lemma~\ref{lem:comm->conj_into_centr}.
But $u \notin \langle a_n \rangle^{G_n}$, as $d \stackrel{G_n}{\not\approx} a_n$, hence   $\C_{G_n}(u) \subseteq H_n \subset N_n$ by \eqref{eq:4}.
The latter would imply that $d \in N_n$, because $N_n$ is normal in $G_n$, contradicting the assumption of this case.

So, in Case 1 we define $L_n$ as the HNN-extension
\begin{equation}\label{eq:case_1}
L_n=\langle G_n, t \,\|\, ta_n t^{-1}=d\rangle.
\end{equation}

\noindent\emph{Case 2:} $d \in N_n$. Since $H_n \cong H_0$ is non-elementary, by \cite[Lemma 3.8]{Olsh} there exists an infinite order
 element $e \in H_n$ such that
 \begin{equation}\label{eq:Case_2-non-comm}
 e \stackrel{G_n}{\not\approx} a_n,~
e \stackrel{G_n}{\not\approx} d \text{ and } \C_{G_n}(e)=\langle e \rangle.
\end{equation}
Thus in Case 2 we define $L_n$ as the HNN-extension
\begin{equation}\label{eq:case_2} L_n=\langle G_n, t \,\|\, te t^{-1}=d\rangle.\end{equation}

Now suppose that we are in one of the above two cases. The group $L_n$ is torsion-free and hyperbolic, and $H_n$ is quasiconvex in $L_n$ by Lemmas~\ref{lem:elem}, \ref{lem:HNN-hyp}. Therefore $Q_n=H_n\cup H_n^{a_n}$ is quasiconvex in $L_n$ by Remark~\ref{rem:props-qc}.
Let $\xi: G_n \to \Z/2\Z$ denote the epimorphism with kernel $N_n$.
Observe that in each of the presentations \eqref{eq:case_1}, \eqref{eq:case_2} above the free letter $t$ conjugates elements which have the same image under $\xi$ in $\Z/2\Z$. Therefore $\xi$ can be extended to
an epimorphism $\widehat{\xi}:L_n \to \Z/2\Z$ by letting $\widehat{\xi}(t)=\overline{0}$. Observe that $M_n=\ker(\widehat{\xi})$ is
the normal closure of $N_n$ and $t$ in $L_n$, and satisfies
\begin{equation}\label{eq:M_n}
|L_n:M_n|=2 \text{ and } M_n \cap G_n=N_n.
\end{equation}


We shall now check that conditions  \eqref{eq:4} and \eqref{eq:5} are satisfied for $i=n$, when $G_n$ is replaced by $L_n$ and $N_n$
is replaced by $M_n$. Note that the associated subgroups
($\langle a_n \rangle$ and $\langle d \rangle$ in Case~1, and $\langle e \rangle$ and $\langle d\rangle$ in Case~2)
of the  HNN-extension $L_n$ are malnormal in $G_n$ by Lemma~\ref{lem:elem}, and their $G_n$-conjugates
can only intersect trivially. Therefore any  malnormal cyclic subgroup of $G_n$ will remain malnormal in $L_n$ by part (a) of  Lemma~\ref{lem:HNN-conj}.
It follows that
\begin{equation}\label{eq:4-L_n}
\C_{L_n}(a_n^l)=\langle a_n \rangle, ~\forall\, l \in \Z\setminus\{0\}, \text{ and } \C_{L_n}(u)=\C_{G_n}(u) \subseteq H_n ,
~\forall\,
u \in H_n \setminus\langle a_n^2\rangle^{H_n}.
\end{equation}

Assume that for some $l \in 2\N$,  $a_n^l$ is conjugate in $L_n$ to an element $h \in H_n$. We will show that these two elements are actually conjugate in $G_n$. Arguing by contradiction, suppose that $h \notin (a_n^l)^{G_n}$. If $h \in (a_n^k)^{G_n}$, for $k \neq l$, then, since $\langle a_n \rangle$ is malnormal in $L_n$, $a_n^k$ cannot be conjugate to $a_n^l$ in $L_n$, contradicting our assumption. Thus $h \notin \langle a_n \rangle^{G_n}$. If we are in Case 1 above, the latter, combined with \eqref{eq:Case_1-non-comm}, implies that $a_n^l \notin h^{L_n}$ by part (b) of  Lemma~\ref{lem:HNN-conj},
which again contradicts the assumption that $h$ is conjugate to $a_n^l$ in $L_n$. Similarly, in Case~2, in view of \eqref{eq:Case_2-non-comm}
and the assumption  that
$a_n \stackrel{G_n}{\not\approx} d$, part (b) of  Lemma~\ref{lem:HNN-conj} shows that $h \notin (a_n^l)^{L_n} $, leading to another contradiction. Thus we have shown that $a_n^l \in h^{G_n}$. After recalling \eqref{eq:5}, we see that the following has been established:
\begin{equation}\label{eq:5-L_n}
\text{if } a_n^l \in h^{L_n}, \text{ for some } l \in 2\N \text{ and } h \in H_n, \text{ then } a_n^l\in h^{H_n}.
\end{equation}

We will now construct the group $G_{n+1}$ as a quotient of $L_n$, using Theorem~\ref{thm:pres_qc}. Denote $F=\langle a_n, b_n^c \rangle \cap N_n \leqslant G_n$
and observe that $|\langle a_n, b_n^c \rangle :F|=2$. Note that the subgroup $\langle a_n^2,b_n^c\rangle \leqslant F$ is non-elementary (otherwise $a_n^{l} = (b_n^k)^c$, for some $l \in 2\N$, $k \in \Z\setminus\{0\}$, which, by \eqref{eq:5}, would imply that $a_n^l \in (b_n^k)^{H_n}$, contradicting the fact that $a_n^2$ and $b_n$ freely generate $H_n$ by \eqref{eq:3}). Therefore $F$ is non-elementary.

Let us prove that $Q_n=H_n \cup H_n^{a_n}$ is small relative to $F$ in $L_n$. Indeed, by Theorem~\ref{thm:small-crit}, it is sufficient to show that $|F:(F \cap H_n^x)|=\infty$ for all $x \in L_n$. Arguing by contradiction, assume that
\begin{equation}\label{eq:fin_ind}
|F:(F \cap H_n^x)|<\infty \text{ for some } x \in L_n.
\end{equation}


Since $a_n^2 \in F$, by \eqref{eq:fin_ind} there must exist $l \in 2\N$ and $h \in H_n$ such that $a_n^l=h^x$. So, in view of \eqref{eq:5-L_n}, we can find
$y \in H_n$ satisfying $a_n^l=h^y$. The latter yields that $h=y^{-1}a_n^l y=x^{-1}a_n^l x$, so that $x y^{-1} \in \C_{L_n}(a_n^l)=\langle a_n \rangle$ by \eqref{eq:4-L_n}. Hence $x = a_n^k y$, for some $k \in \Z$, and $H_n^x=a_n^k H_n a_n^{-k}$. Since $F$ is normalized by $a_n$, by definition, the latter,
combined with \eqref{eq:fin_ind}, implies that $|F:(F \cap H_n)|<\infty$ and $|F:(F \cap H_n^{a_n})|<\infty$, hence
\begin{equation}\label{eq:fin_ind-2}
|F:(F \cap (H_n \cap H_n^{a_n}))|<\infty.
\end{equation}
But, by the injectivity of $\psi_n$ on $Q_0$ \eqref{eq:3} and part (ii) of Lemma~\ref{lem:H-props}, we have
 \[H_n\cap H_n^{a_n}=\psi_n(H_0) \cap \psi_n(H_0^{a_0})=\psi_n(H_0 \cap H_0^{a_0})=\psi_n(\langle a_0^2 \rangle)=\langle a_n^2\rangle.\]
 Thus \eqref{eq:fin_ind-2} yields $|F:(F \cap \langle a_n^2 \rangle)|<\infty$, contradicting the fact that $F$ is a non-elementary subgroup of $G_n$. Hence we have proved that $Q_n$ is small relative to $F$ in $L_n$.

Therefore we can apply Theorem~\ref{thm:pres_qc} to $L_n$, $F$, $Q_n$ and $U=\{b_n,t\}$. Let $w_1,w_2 \in F$ be the elements from the claim of this theorem, let $K_n$ denote the normal closure of the elements $b_nw_1$ and $tw_2$ in $L_n$, let $G_{n+1}=L_n/K_n$,
let  $\eta_{n+1}: L_n \to G_{n+1}$ be the natural epimorphism
and $\phi_{n+1}:G_n\to G_{n+1}$ be the restriction of $\eta_{n+1}$ to $G_n\leqslant L_n$. Denote $\psi_{n+1}=\phi_{n+1}\circ \psi_n:G_0\to G_{n+1}$,
$a_{n+1}=\phi_{n+1}(a_n)=\psi_{n+1}(a_0)$, $b_{n+1}=\phi_{n+1}(b_n)=\psi_{n+1}(b_0)$, etc., as in the beginning of the proof, for $i=n+1$. Then the following conditions will be satisfied by Theorem~\ref{thm:pres_qc}.
\begin{enumerate}[(a)]
  \item $G_{n+1}$ is a torsion-free hyperbolic group. Moreover, since $L_n=\langle a_n,b_n,t\rangle$ and
        $\eta_{n+1}(t)=\eta_{n+1}(w_2^{-1})\in \eta_{n+1}(G_n)$, we see that $\eta_{n+1}(G_n)=G_{n+1}$, i.e., $\phi_{n+1}:G_n \to G_{n+1}$ is surjective and
         $G_{n+1}=\langle a_{n+1},b_{n+1} \rangle$. Hence \eqref{eq:1} holds for $i=n+1$.
  \item The restriction of $\phi_{n+1}$ to $Q_n$ is injective and $H_{n+1}=\phi_{n+1}(H_n)$, $Q_{n+1}=\phi_{n+1}(Q_n)$ are quasiconvex in $G_{n+1}$.
         Thus \eqref{eq:3} holds for $i=n+1$.
 \item For all $x \in H_n$, $\C_{G_{n+1}}(\eta_{n+1}(x))=\eta_{n+1}(\C_{L_n}(x))$. In view of \eqref{eq:4-L_n} and the injectivity of $\eta_{n+1}$ on $H_{n}$,
        this implies that  \eqref{eq:4} holds for $i=n+1$.
 \item $\eta_{n+1}$ preserves conjugacy on $H_n$. In view of \eqref{eq:5-L_n},
        this implies that  \eqref{eq:5} holds for $i=n+1$.
\end{enumerate}

Observe that $b_{n+1}=\phi_{n+1}(b_n)=\phi_{n+1}(w_1^{-1}) \in \phi_{n+1}(F)$ in $G_{n+1}$. Since $F \subset \langle a_n,b_n^c \rangle$ and
$G_{n+1}=\langle a_{n+1},b_{n+1} \rangle$, we can conclude that
$G_{n+1}=\phi_{n+1}(\langle a_{n},b_{n}^c \rangle)=\langle a_{n+1},b_{n+1}^{g_{n+1}} \rangle$, where
$g_{n+1}=\phi_{n+1}(c)=\psi_{n+1}(f_{n+1})$. Thus \eqref{eq:7} holds for $i=n+1$, and it remains to check that $G_{n+1}$ satisfies
\eqref{eq:2} and \eqref{eq:6}.

Note that $N_{n+1}=\phi_{n+1}(N_n)$ is a normal subgroup in $G_{n+1}$ of index at most $2$ because
$\phi_{n+1}$ is surjective and $|G_n:N_n|=2$.
Obviously $N_{n+1} \subseteq \eta_{n+1}(M_n)$. The opposite
inclusion follows from the fact that $M_n$ is generated, as a normal subgroup of $L_n$, by $N_n$ and $t$, and
\begin{equation}\label{eq:eta(t)_in_N}
\eta_{n+1}(t)=\eta_{n+1}(w_2^{-1}) \in \eta_{n+1}(F) \subseteq \eta_{n+1}(N_n)=N_{n+1}.
\end{equation}
 Hence $N_{n+1}=\eta_{n+1}(M_n)$ in $L_n$. Recall that, by construction, $b_n,t \in M_n$ and
$w_1,w_2 \in F \subset M_n$ in $L_n$, yielding that $\ker \eta_{n+1}=K_n \subseteq M_n$. Therefore $G_{n+1}/N_{n+1} \cong L_n/M_n\cong \Z/2\Z$, i.e., $N_{n+1}$
has index $2$ in $G_{n+1}$. It follows that condition  \eqref{eq:2} holds for $i=n+1$.

By construction, $\psi_n(f_{n+1})=c$ either belongs to $\langle a_n \rangle^t $ (Case 1) or to  $ H_n^t$ (Case~2) in $L_n$. Therefore condition
\eqref{eq:6} for $i=n+1$ follows from \eqref{eq:eta(t)_in_N}. This concludes the inductive step in our argument, and thus finishes the construction of
a sequence of hyperbolic groups $G_0,G_1,G_2,\dots$,  together with epimorphisms $\phi_i:G_{i-1}\to G_i$, $i \in \N$, satisfying properties
\eqref{eq:1}--\eqref{eq:7} above.

We can now define the group $G$ as the direct limit of the sequence $(G_{i-1},\phi_i)_{i \in \N}$. In other words, $G=G_0/K$, where
$K=\bigcup_{i=1}^\infty \ker \psi_i$. Let $\psi:G_0 \to G$ be the natural epimorphism, so that $\ker\psi=K$ and $\psi$ factors through $\psi_i:G_0 \to G_i$, for each $i \in \N$.

Let $a=\psi(a_0)$, $b=\psi(b_0)$, $H=\psi(H_0)=\langle a^2,b\rangle \leqslant G$ and
$N=\psi(N_0)=\langle a^2,b,b^a\rangle\lhd G$. We will now check that $G$, $H$ and $N$ satisfy the properties from the claim of Theorem~\ref{thm:fin_ind}.

The group $G$ is torsion-free as a direct limit of torsion-free groups $G_i$.
The index $|G:N|=2$ because, by \eqref{eq:2}, $\ker\psi_i \subseteq N_{0}$ for
all $i \in \N$. By \eqref{eq:3}, each $\psi_i$ is injective on $Q_0$, hence the same is true for $\psi$. It follows that
$H \cap H^a=\psi(H_0 \cap H_0^{a_0})=\langle a^2 \rangle $ (by Lemma~\ref{lem:H-props}), and $H \cong H_0$ is freely generated by $a^2$ and $b$.
If $|N:H|<\infty$ then $|N:H^a|<\infty$, since $a$ normalizes $N$, hence $|N:(H \cap H^a)|<\infty$, which would mean that $N$ is virtually cyclic, contradicting the
fact that it contains a non-abelian free subgroup $H$. Therefore we can deduce that $|N:H|=\infty$.

Given any $x,y \in G$, the subgroup $\langle a^x,b^y\rangle$ is conjugate to the subgroup $\langle a,b^g \rangle$, where $g=x^{-1}y$. By construction,
$g=\psi(f_i)$, for some $i \in \N\cup\{0\}$, and \eqref{eq:7} implies that $G_i$ is generated by $a_i=\psi_i(a_0)$ and $b_i^{g_i}=\psi_i(b_0^{f_i})$.
It follows that $G=\langle a,b^g \rangle$, since $\psi$ factors through $\psi_i$. Hence $G=\langle a^x,b^y\rangle$ and
$\{a,b\}$ is an invariable generating set for $G$.

It remains to show that every element of $N$ is conjugate to an element of $H$ in $N$.
Indeed, \eqref{eq:6} easily implies that $G=H^N \cup \langle a \rangle^N$, and since
\[\langle a \rangle^N \cap N=(\langle a \rangle \cap N)^N=\langle a^2 \rangle ^N \subset H^N,\]
we can conclude that $N=H^N$, as required. This finishes the proof of Theorem~\ref{thm:fin_ind}.
\end{proof}

\begin{rem}\label{rem:cyc_centr-1} Let $G$ be the group constructed in the proof of Theorem~\ref{thm:fin_ind}. Then centralizers of non-trivial elements in $G$ are cyclic.
\end{rem}

Indeed, for any $g \in G\setminus\{1\}$, $g$ has infinite order and $g^2 \in N$, so, by claim (iv)
of the theorem, $g^2$ is conjugate to some element $u \in H\setminus\{1\}$. If $\C_G(u) \subseteq H$, then $\C_G(u)=\C_H(u)$ is cyclic, as $H$ is free.
Otherwise, \eqref{eq:4} implies that $u \in \langle a^2 \rangle^H$ and $\C_G(u)$ is conjugate to
 $\langle a \rangle$. Thus $\C_{G}(u)$ is cyclic, so the same is true for $\C_G(g^2)$. Therefore $\C_G(g)$ must also be cyclic, as a subgroup of $\C_G(g^2)$
(in fact, in this case $\C_G(g)=\C_G(g^2)$).

\begin{rem}\label{rem:add_props_Thm1}
The construction of the group $G$ and its index two subgroup $N$ in Theorem~\ref{thm:fin_ind} is fairly flexible, and many additional properties can be
achieved:
\begin{itemize}

  \item by adding a sufficiently large finite subset of $G_i$ to $Q_i$ at each step, one can ensure that $G$ is lacunary hyperbolic (see \cite{OOS} for the definition and properties of lacunary hyperbolic groups);
  \item by modifying $N_0$ and $G_0$ with the help of Theorem~\ref{thm:pres_qc}, one can arrange $N$ to be a quotient of any given torsion-free non-elementary hyperbolic group; in particular, $N$ and $G$ can be made to satisfy Kazhdan's property (T);
  \item by changing the construction at each step, one can achieve even more and ensure that $N$ is a common quotient of all non-cyclic torsion-free hyperbolic groups;

  \item it should be possible to extend the method of proof to produce examples of FIG groups with non-IG subgroups of index $k$, for every $k \ge 2$, though additional technical modifications will be necessary.
\end{itemize}
\end{rem}

\section{Finitely generated IG groups that are not FIG}
\begin{proof}[Proof of Theorem~\ref{thm:IG_not_FIG}]
As before, the desired group $G$ will be constructed as a direct limit of hyperbolic groups $G_i$, $i \ge 0$.
Let $G_0$ be the free group of rank $3$, freely generated by $\{a_0,b_0,c_0\}$, and set $Q_0=H_{00}=\{1\} \leqslant G_0$.
Let $S_0=\emptyset,S_1=\{1\},S_2, \dots$ be an enumeration of all finite subsets of $G_0$, let
$Y_0=\{1\},Y_1=\langle a_0 \rangle,Y_2=\langle a_0,b_0\rangle, \dots$ be an enumeration of all finitely generated subgroups of $G_0$.

Let $G_1$ be a copy of $G_0$, with a fixed isomorphism $\phi_1:G_0 \to G_1$. Define $H_{10}=\phi_1(H_{00})=\{1\}$,
$H_{11}=\langle b_1,c_1 \rangle$, where
$b_1=\phi_1(b_0) \in G_1$, $c_1=\phi_1(c_0) \in G_1$, and let $Q_1=H_{10} \cup H_{11}=\phi_1(Q_0) \cup H_{11} \subset G_1$.

Now suppose that for some $n \ge 1$ and each $i=1,\dots,n$ we have already constructed a group $G_i$, an epimorphism $\phi_i:G_{i-1} \to G_i$, and
subgroups $H_{ij} \leqslant G_i$, $j=0,1,\dots,i$, such that $H_{ij}=\phi_i(H_{i-1,j})$, whenever $j\in \{0,\dots,i-1\}$.
We let $\psi_i : G_0 \to G_i$ be the epimorphism defined by $\psi_i=\phi_i \circ \dots \circ \phi_1$,
set $T_i=\psi_i(S_i) \subset G_i$, $Z_i=\psi_i(Y_i) \leqslant G_i$,
and $Q_i=\bigcup_{j=0}^i H_{ij}=\phi_i(Q_{i-1}) \cup H_{ii} \subseteq G_i$.

Arguing by induction, we assume that the following conditions hold for every $i=1,\dots, n$:
\begin{flalign}
&G_i \text{ is a non-elementary torsion-free hyperbolic group};\label{eq:21}&&\\[2ex]
&\phi_i \text{ is injective on } Q_{i-1}, \text{ and } H_{ij} \text{ is free and quasiconvex in } G_i,~j=0,\dots,i; \label{eq:22}&&\\[2ex]
&\text{$\phi_i$ preserves conjugacy and centralizers on $Q_{i-1}$}; \label{eq:23}&& \\[2ex]
&\text{$H_{ij}$ is malnormal in $G_i$, for all $j=0,\dots,i$}; \label{eq:24}&& \\[2ex]
&\text{for all } g \in G_i, ~ 0 \le j<k \le i,\text{ the intersection } H_{ij}^g \cap H_{ik} \text{ is cyclic}; \label{eq:25}&&\\[2ex]
&|G_i:H_{ii}|=\infty \text{ and there exists } x \in H_{ii}\setminus\{1\} \text{ such that } \langle x \rangle \cap \left(\bigcup_{j=1}^{i-1}H_{ij}^{G_i}\right)=\{1\}; \label{eq:26}&&\\[2ex]
&T_i \subseteq H_{ii}^{G_i}; \label{eq:27}&& \\[2ex]
&\text{one of the following conditions is satisfied:} \label{eq:28}&&\\
&\text{\textbullet}~ Z_i=G_i, \text{ or } &&\notag\\
&\text{\textbullet}~ Z_i \text{ is cyclic, or }&&\notag\\
&\text{\textbullet}~ |Z_i:(Z_i\cap H_{ij}^u)|<\infty \text{ for some } j=0,1,\dots,i, \text{ and some } u \in G_i. \notag&&
\end{flalign}

Note that for $i=1$ the group  $G_1$, the epimorphism $\phi_1:G_0 \to G_1$, and the subgroups $H_{10},H_{11} \leqslant G_1$ satisfy all of the above properties by definition,
so the base of induction has been established. Our aim now is to construct a group $G_{n+1}$, an epimorphism $\phi_{n+1}:G_n \to G_{n+1}$ and a subgroup $H_{n+1,n+1} \leqslant G_{n+1}$,
enjoying properties \eqref{eq:21}--\eqref{eq:28} for $i=n+1$.
As before, our method involves an intermediate torsion-free hyperbolic group $L_n$ and a free quasiconvex subgroup $X \leqslant L_n$, such that $G_n \leqslant L_n$ and $\psi_n(S_{n+1})$ is pointwise conjugate into $X$ in $L_n$. The group $G_{n+1}$ will be obtained as a small cancellation quotient of $L_n$, and the subgroup $H_{n+1,n+1} \leqslant G_{n+1}$ will be defined as the image of $X$ in $G_{n+1}$.

Recall that $\psi_n:G_0 \to G_n$ is the epimorphism $\phi_n\circ \dots \circ \phi_1$, and suppose that
$\psi_n(S_{n+1})\setminus\{1\}=\{s_1,\dots,s_l\} \subset G_n$, for some
$l \in \N \cup \{0\}$. Let $d_1,\dots,d_l \in G_n$ be the elements satisfying $\C_{G_n}(s_k)=\langle d_k \rangle$, for $k=1,\dots,l$ (cf. Lemma~\ref{lem:elem}).
Let $X$ be a free group of rank $l+1$, with free basis $\{x_0,x_1,\dots,x_l\}$. We define the group $L_n$ by the following presentation:
\begin{equation}\label{eq:Ln-2}
L_n=\langle G_n,x_0,\dots,x_l,t_1,\dots,t_l \,\|\, t_k x_k t_k^{-1}=d_k, ~k=1,\dots,l \rangle.
\end{equation}

Since $G_n$ is torsion-free, each of $d_1,\dots,d_l$ has infinite order in $G_n$, so $L_n$ is an $l$-fold HNN-extension of the free product $G_n*X$ with associated cyclic subgroups. The group $G_n*X$ is hyperbolic, as a free product of hyperbolic groups, and $G_n$, $X$ are both quasiconvex in it (for example,
by \cite[Lemma~1.2]{Min-G-sbgps}). After applying Lemma~\ref{lem:HNN-hyp} $l$ times, we see that $L_n$ is a torsion-free hyperbolic group and $G_n$, $X$ are quasiconvex in it. Therefore $H_{nj}$ is a quasiconvex subgroup of $L_n$, for each $j=0,\dots,n$, by part (5) of Remark~\ref{rem:props-qc}.

Now, $G_n$ and $X$ are malnormal in $G_n*X$, being free factors, hence $\langle d_k\rangle$ and $H_{nj}$ are malnormal in $G_n*X$, for all $k=1,\dots,l$ and
 $j=0,\dots,n$, by \eqref{eq:24}.
Observe that for any group $A$ and any cyclic group $\langle x \rangle$, $\langle x \rangle$ and every malnormal subgroup of $A$ are malnormal in the free product $A*\langle x \rangle$; moreover,  no non-trivial element of $A$ is conjugate to an element of $\langle x \rangle$. Therefore we can apply claim (a) of Lemma~\ref{lem:HNN-conj} $l$ times to show that $H_{nj}$ and
$X$ are malnormal in $L_n$, for every $j=0,\dots,n$. It follows that $\C_L(h) \subseteq H_{nj}$ for all $h \in H_{nj}\setminus\{1\}$,
$j=0,\dots,n$; in particular,
\begin{equation}\label{eq:central-pres-in-Ln}
\C_L(h)=\C_{G_n}(h), \text{ for every } h \in Q_n\setminus\{1\}.
\end{equation}

From the presentation \eqref{eq:Ln-2}
it is easy to see that there is a retraction $\rho:L_n \to G_n$, such that the restriction of $\rho$ to $G_n$ is the
identity map, $\rho(x_0)=1$, $\rho(x_k)=d_k$ and $\rho(t_k)=1$ for all $k=1,\dots, l$. It follows that
\begin{equation}\label{eq:retr-conj}
\text{two elements of $G_n$ are conjugate in $L_n$ if and only if they are conjugate in $G_n$.}
\end{equation}

For the next part it will be more convenient to think of $L_n$ as the fundamental group of a graph of groups with two vertices $v_1,v_2$ and
$l+1$ edges $e_0,\dots,e_l$, joining these two vertices: see Figure~\ref{fig:graph_of_gps}.
\begin{figure}[ht]
  \begin{center}
   \includegraphics{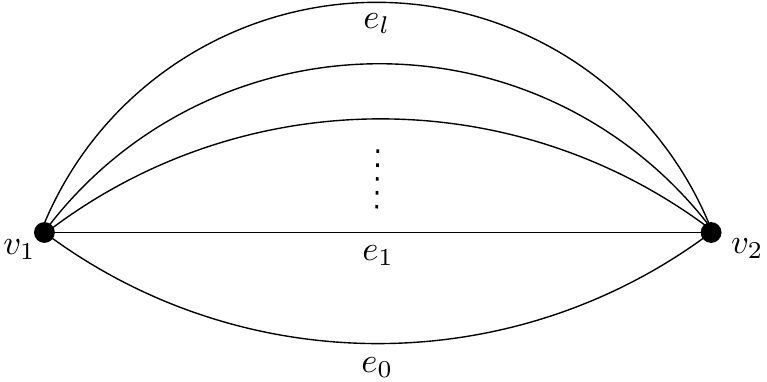}
  \end{center}
\caption{The underlying graph for the splitting of $L_n$.}\label{fig:graph_of_gps}
\end{figure}
The vertex groups in this splitting will be $G_n$ for $v_1$ and $X$ for $v_2$, the edge group for $e_0$ will be trivial, and the edge groups for $e_1, \dots,e_l$
will be infinite cyclic. This gives rise to an action of $L_n$ on the Basse-Serre tree $\mathcal T$ corresponding to this splitting. Vertex stabilizers for this
action are conjugates of $G_n$ or $X$ in $L_n$ and edge stabilizers are cyclic. Hence the intersection of stabilizers of two distinct vertices of $\mathcal T$
is cyclic, which immediately yields the following observations:
\begin{equation}\label{eq:intersec_G-X}
G_n^u \cap X \text{ is cyclic, for all } u \in L_n, \text{ and }
\end{equation}
\begin{equation}\label{eq:intersec_G-G}
\text{ if } w \in L_n \setminus G_n \text{ then } G_n^w \cap G_n \text{ is cyclic}.
\end{equation}
The last observation, combined with \eqref{eq:25}, implies
\begin{equation}\label{eq:intersec_H-H}
\text{for all } u \in L_n, ~ 0 \le j<k \le n,\text{ the intersection } H_{nj}^u \cap H_{nk} \text{ is cyclic}.
\end{equation}

Let $Q=Q_n \cup X=\bigcup_{j=1}^n H_{nj} \cup X \subset L_n$ and $Z=\psi_n(Y_{n+1}) \leqslant G_n$.
If $Z$ is non-elementary and $Q$ is small relative to $Z$ in $L_n$,
then set $F=Z$. Otherwise, set $F=G_n$. Recall that $G_n$ is torsion-free and non-elementary (by \eqref{eq:21}), hence  \eqref{eq:26} implies that
$|G_n:H_{nj}|=\infty$, for each $j=0,\dots,n$. Combined with \eqref{eq:intersec_G-G}, this yields that $|G_n:(G_n \cap w^{-1}H_{nj}w)|=|G_n^w:(G_n^w \cap H_{nj})|=\infty$ for all $w \in L_n$
and $j=0,\dots,n$. On the other hand, $|G_n:(G_n \cap u^{-1}X u)|=\infty$ for all $u \in L_n$ by \eqref{eq:intersec_G-X}. Therefore we can apply Theorem~\ref{thm:small-crit} to conclude that
$Q$ is small relative to $G_n$ in $L_n$.
 Thus, in any case, $F$ is non-elementary, $Q$ is small relative to $F$ in $L_n$ and $F \subseteq G_n$.

Since $Q$ is quasiconvex in $L_n$ (by part (4) of Remark~\ref{rem:props-qc}), we can apply Theorem~\ref{thm:pres_qc} to $L_n$, $F$, $Q$ and
$U=\{x_0,\dots,x_l,t_1,\dots,t_l\}$. Let $G_{n+1}$ denote the resulting quotient of $L_n$, let $\eta_{n+1}:L_n \to G_{n+1}$ be the natural epimorphism, which identifies each element of $U$ with some element of $F$ in $G_{n+1}$. Since $F \subseteq G_n$ and $L_n=\langle G_n,U\rangle$, we see that $G_{n+1}=\eta_{n+1}(L_n)=\eta_{n+1}(G_n)$, i.e.,
 the restriction $\phi_{n+1}:G_n \to G_{n+1}$,
of $\eta_{n+1}$ to $G_n$, is surjective. Set $H_{n+1,n+1}=\eta_{n+1}(X) \leqslant G_{n+1}$, $H_{n+1,j}=\phi_{n+1}(H_{nj}) \leqslant G_{n+1}$ if $0 \le j \le n$,
and $Q_{n+1}=\bigcup_{j=0}^{n+1} H_{n+1,j}$.
Let $\psi_{n+1}=\phi_{n+1} \circ \psi_n:G_0 \to G_{n+1}$,
$T_{n+1}=\psi_{n+1}(S_{n+1}) \subset G_{n+1}$, $Z_{n+1}=\psi_{n+1}(Y_{n+1})\leqslant G_{n+1}$.

Theorem~\ref{thm:pres_qc} implies that all of the following hold.
\begin{enumerate}[(a)]
  \item $G_{n+1}$ is a non-elementary torsion-free hyperbolic group, so \eqref{eq:21} is true for $i=n+1$.
  \item The restriction of $\eta_{n+1}$ to $Q=Q_n\cup X$ is injective, and $H_{n+1,j}$ is a quasiconvex subgroup of $G_{n+1}$, for all $j=0,\dots,n+1$.
    It follows that $\phi_n$ is injective on $Q_n$, and $H_{n+1,j} \cong H_{nj}$, $0 \le j \le n$, $H_{n+1,n+1} \cong X$ are free subgroups of $G_{n+1}$.
    Thus \eqref{eq:22} holds for $i=n+1$.
    \item $\eta_{n+1}$ preserves centralizers on $Q$, hence, by \eqref{eq:central-pres-in-Ln}, $\phi_{n+1}$ preserves centralizers on $Q_n$.
 \item $\eta_{n+1}$ preserves conjugacy on $Q$:
    \begin{equation}\label{eq:conj_in_Gn+1}
    \text{if $g,h \in Q_{n} \cup X$, then $\eta_{n+1}(h) \in \eta_{n+1}(g)^{G_{n+1}}$  in $G_{n+1}$ implies $h \in g^{L_{n}}$ in $L_{n}$}.
    \end{equation}
    Combined with \eqref{eq:retr-conj}, this immediately shows that $\phi_{n+1}$ preserves conjugacy on $Q_n$.
\end{enumerate}

Claims (c) and (d) above imply that \eqref{eq:23} holds for $i=n+1$. We can also apply Lemma~\ref{lem:malnorm_pres}
to deduce that for every $j=0,\dots,n+1$, $H_{n+1,j}$ is malnormal in $G_{n+1}$, i.e., $\eqref{eq:24}$ is satisfied for $i=n+1$.

The fact that \eqref{eq:25} is satisfied in $G_{n+1}$ follows from Lemma~\ref{lem:cyc_intersec_preserved}. Indeed, for any $0 \le j<k\le n+1$,
$H_{n+1,j}$ and $H_{n+1,k}$ are images, under $\eta_{n+1}$, of subgroups contained in $Q_n \cup X$ in $L_n$, so the assumptions of Lemma~\ref{lem:cyc_intersec_preserved} are satisfied by \eqref{eq:conj_in_Gn+1}, \eqref{eq:intersec_H-H} and \eqref{eq:intersec_G-X}.

By the above argument, $H_{n+1,n+1} \cap H_{n+1,1}$ is cyclic, and since $H_{n+1,1} \cong H_{11} $ (by \eqref{eq:22}) is a free group of rank $2$,
we can immediately deduce that $H_{n+1,n+1}$ must have infinite index in $G_{n+1}$. Now, since the generator $x_0$, of $X$, does not
participate in any of the defining relations from \eqref{eq:Ln-2}, we see that $L_n \cong B*\langle x_0\rangle$, where
$B=\langle G_n,x_1,\dots,x_l,t_1,\dots,t_l\rangle \leqslant L_n$. In particular, $\langle x_0 \rangle \cap H_{nj}^{L_n}=\{1\}$, for all $j=0,\dots,n$, which can
be combined with \eqref{eq:conj_in_Gn+1} to give $\langle \eta_{n+1}(x_0) \rangle \cap H_{n+1,j}^{G_{n+1}}=\{1\}$, for all $j=0,\dots,n$.
As $\eta_{n+1}(x_0)$ is an infinite order element of $\eta_{n+1}(X)=H_{n+1,n+1}$, we can conclude that \eqref{eq:26}
is satisfied for $i=n+1$.

The inclusion \eqref{eq:27} for $i=n+1$ is an immediate consequence of the construction of $L_n$ as the HNN-extension \eqref{eq:Ln-2}, and it remains to verify
that $G_{n+1}$ satisfies \eqref{eq:28} for $i=n+1$, where $Z_{n+1}=\psi_{n+1}(Y_{n+1})=\eta_{n+1}(Z)$.
But this is indeed the case due to our choice of the subgroup $F \leqslant G_n$ above and Theorem~\ref{thm:small-crit}.

Thus we have checked that $G_{n+1}$, $\phi_{n+1}:G_n\to G_{n+1}$ and $H_{n+1,n+1}\leqslant G_{n+1}$ satisfy all of the properties \eqref{eq:21}--\eqref{eq:28} for $i=n+1$, which
completes our inductive construction.

Let $G$ be the direct limit of the sequence $(G_{i-1},\phi_{i})_{i \in \N}$. Then $G$ can be described as the quotient of $G_0$ by
$K=\bigcup_{i=1}^\infty \ker \psi_i$, and the natural epimorphism $\psi:G_0 \to G$ factors through $\psi_i:G_0 \to G_i$, for each $i \in \N$.
Let $\xi_i:G_i \to G$ denote the resulting epimorphism such that $\psi=\xi_i \circ \psi_i$, for each $i \in \N$.
A commutative diagram involving epimorphisms between the groups $G_0$, $G_i$, $G_{i+1}$ and $G$, $i \in \N$, discussed so far, is depicted in Figure~\ref{fig:comm_diag} below.
\begin{figure}[!ht]
  \begin{center}
\begin{tikzcd}
G_0 \arrow[r, "\psi_i"'] \arrow[rr, "\psi_{i+1}", bend left] \arrow[rd, "\psi"'] & G_i \arrow[r, "\phi_{i+1}"'] \arrow[d, "\xi_i"] & G_{i+1} \arrow[ld, "\xi_{i+1}"] \\
                                                                                 & G                                               &
\end{tikzcd}
  \end{center}
\caption{A commutative diagram of maps between $G_0$, $G_i$, $G_{i+1}$ and $G$.}\label{fig:comm_diag}
\end{figure}

\begin{lemma} \label{lem:conj_in_Q_i<->conj_in_G} For all  $j\ge 0$  the epimorphism $\xi_j :G_j \to G$ preserves conjugacy on $Q_j$.
\end{lemma}

\begin{proof}[Proof of Lemma \ref{lem:conj_in_Q_i<->conj_in_G}]
Suppose $\xi_j(h)=\xi_j(g)^q$, for some $g,h \in Q_j$ and some $q \in G$. Let $g_0,h_0,q_0 \in G_0$ be
arbitrary preimages of the elements $g,h,q$ in $G_0$, so that $g=\psi_j(g_0)$, $h=\psi_j(h_0)$ and $q=\psi(q_0)$.
Then $h_0^{-1}g_0^{q_0} \in \ker \psi=\bigcup_{n=1}^\infty \ker\psi_n$, so
$h_0^{-1}g_0^{q_0} \in \ker \psi_i$, for some $i \in \N$. Moreover, we can assume that $i>j$, as $\ker\psi_i \subseteq \ker\psi_{i+1}$ for all $i$.
Hence $g'=\psi_i(g_0)$ is conjugate to $h'=\psi_i(h_0)$ in $G_i$.
Let $\varphi: G_j \to G_i$ be the epimorphism defined by $\varphi= \phi_i \circ \dots \circ  \phi_{j+1}$, and
observe that $g'=\varphi(g)$ and $h'=\varphi(h)$, so that $g',h' \in \varphi(Q_j)$ in $G_i$. Since $h' \in (g')^{G_i}$ we can apply
\eqref{eq:23} $i-j$ times to conclude that $h \in g^{G_j}$ in $G_j$, as required.
\end{proof}

The next statement also follows from \eqref{eq:23} and can be proved similarly to Lemma~\ref{lem:conj_in_Q_i<->conj_in_G}.
\begin{lemma}\label{lem:xi_j-pres_central} For all  $j\ge 0$  the epimorphism $\xi_j :G_j \to G$ preserves centralizers on $Q_j$.
\end{lemma}

We can now start checking that $G$ satisfies all the properties listed in the statement of Theorem~\ref{thm:IG_not_FIG}.
The group $G$ is indeed $2$-generated, because it is a quotient of $G_2$, and $G_2=Z_2$ is a quotient of
$Z=\langle a_1,b_1\rangle=\phi_1(Y_2)$, by construction ($Z$ is non-elementary and $Q_1=\langle b_1,c_1 \rangle$ is small relative to
$Z$ in $L_1$, as $\{a_1,b_1,c_1\}$ is a free generating set of $G_1$).
Since $G_i$ is torsion-free, for each $i \in \N$, $G$ is also torsion-free.

Observe that $\xi_j(H_{jj})=\xi_i(H_{ij})$ in $G$, whenever $0 \le j \le i$, and  denote this subgroup $H_j \leqslant G$. The epimorphism $\xi_j$
is injective on $H_{jj}$ by \eqref{eq:22}, so $H_j \cong H_{jj}$ is finitely generated and free.
Moreover, in view of \eqref{eq:24} and Lemmas~\ref{lem:conj_in_Q_i<->conj_in_G}, \ref{lem:xi_j-pres_central}, we can apply Lemma~\ref{lem:malnorm_pres} to conclude that
$H_j$ is malnormal in $G$. Thus claim (i) of Theorem~\ref{thm:IG_not_FIG} has been established.

Now, by \eqref{eq:26}, for each $i \in \N$ there is $x \in H_{ii}\setminus \{1\}$ such that $\langle x \rangle \cap H_{ij}^{G_i}=\{1\}$ in $G_i$,
whenever $0 \le j<i$. Set $r=\xi_i(x) \in H_i$, then $r \neq 1$ in $G$, by \eqref{eq:22}, and
$\langle r \rangle \cap H_j^G=\{1\}$, provided $0 \le j<i$, by Lemma~\ref{lem:conj_in_Q_i<->conj_in_G}. This proves claim (ii) of Theorem~\ref{thm:IG_not_FIG}.

If $S$ is an arbitrary finite subset of $G$, then $S=\psi(S_i)$, for some $i \in \N\cup\{0\}$, and $T_i=\psi_i(S_i)$ is pointwise conjugate into $H_{ii}$ in $G_i$, by \eqref{eq:27}. Consequently,
$S=\xi_i(T_i)$ is pointwise conjugate into the subgroup $H_i=\xi_i(H_{ii})$ in $G$, which establishes claim (iii) of Theorem~\ref{thm:IG_not_FIG}.

Claim (iv) of the theorem follows from the lemma below.

\begin{lemma}\label{lem:Hi_inter_Hj} For all $i,j \ge 0$, $i \neq j$, and all $h \in G$, the intersection $H_j^h \cap H_i$ is cyclic in $G$.
\end{lemma}
\begin{proof}[Proof of Lemma~\ref{lem:Hi_inter_Hj}]
Without loss of generality, suppose that $j<i$. By construction, $H_j=\xi_i(H_{ij})$ and $H_i=\xi_i(H_{ii})$, where
$H_{ij}$ and $H_{ii}$ are quasiconvex subgroups of the hyperbolic group $G_i$ (see \eqref{eq:21} and \eqref{eq:22})),
and $H_{ij} \cup H_{ii} \subseteq Q_i$. Recalling Lemma~\ref{lem:conj_in_Q_i<->conj_in_G}
and \eqref{eq:25}, we see that all the assumptions of Lemma~\ref{lem:cyc_intersec_preserved} are satisfied, whence $H_j^h \cap H_i$ must be cyclic
for any $h \in G$.
\end{proof}

Thus it remains to prove claim (v). The next lemma will be used for this.

\begin{lemma}\label{lem:proper-fg_conj_into_Hi} If $N $ is a non-cyclic proper finitely generated subgroup of $G$,
 then there exists $j \in \N \cup \{0\}$ and $v \in G$ such that $N \subseteq H_j^v$.
\end{lemma}

\begin{proof}[Proof of Lemma \ref{lem:proper-fg_conj_into_Hi}]
By construction, there must exist $i\in \N$ such that $N=\psi(Y_i)=\xi_i(Z_i)$. Since $N$ is proper and non-cyclic in $G$,
$Z_i$ will be a proper non-cyclic
subgroup of $G_i$. In view of \eqref{eq:28}, the latter implies that $|Z_i:(Z_i \cap H_{ij}^u)|< \infty$ for some $j=0,1,\dots,i$, and some $ u \in G_i$. Hence
$|N:(N \cap H_j^v)|<\infty$ in $G$, where $v=\xi_i(u) \in G$, so $H_j^v$ contains a finite index normal subgroup of $N$. Since $G$ is torsion-free
and $H_j^v$ is malnormal, $N$ must be contained in $H_j^v$.
\end{proof}

To finish the proof of Theorem \ref{thm:IG_not_FIG}, it remains to show that every proper subgroup $M<G$ is contained in a conjugate of some $H_j$,
$j \in\N\cup\{0\}$. If $M$ is cyclic, this follows from claim (iii) of the theorem, so we can suppose that $M$ is non-cyclic. Take any $x \in M\setminus\{1\}$.
By claim (iii), $x \in H_k^G$, for some $k \in \N\cup \{0\}$, and, after replacing $M$ with a conjugate, we can assume that $x \in H_k$. Then $\C_G(x)$
is cyclic by Lemma~\ref{lem:xi_j-pres_central} and Lemma~\ref{lem:elem}. Since $M$ is non-cyclic, it cannot be contained in $\C_G(x)$, so
there must exist $y \in M$ such that $N=\langle x, y \rangle\leqslant M$ is non-abelian. Hence $N \subseteq H_j^v$ for some $j\in \N\cup\{0\}$ and some $v \in G$, by Lemma~\ref{lem:proper-fg_conj_into_Hi}.

Now, consider any $z \in M$. The subgroup $P=\langle x,y,z\rangle$ is non-cyclic and proper in $G$, so, according to Lemma~\ref{lem:proper-fg_conj_into_Hi},
there exist $i \in \N\cup \{0\}$ and $u \in G$ such that $P \subseteq H_i^u$. Observe that $N \subseteq H_j^v \cap H_i^u$ is non-cyclic, hence $i=j$,
by Lemma~\ref{lem:Hi_inter_Hj}, and $H_j^u=H_j^v$, as $H_j$ is malnormal in $G$. It follows that $z \in P \subseteq H_j^v$, whence $M \subseteq H_j^v$, as required.

Thus we have shown that $G$ satisfies claim (v), so the proof of Theorem~\ref{thm:IG_not_FIG} is complete.
\end{proof}

\begin{rem}\label{rem:cyc_centr-2} It is easy to see that the group $G$ constructed in  Theorem~\ref{thm:IG_not_FIG} satisfies the following properties:
\begin{itemize}
  \item every proper subgroup of $G$ is free;
  \item $G$ is simple;
  \item the centralizers of non-trivial elements are cyclic.
\end{itemize}
\end{rem}

Similarly to Remark~\ref{rem:add_props_Thm1}, with extra work we can impose additional properties on $G$, such as lacunary hyperbolicity and being a quotient of any given (or, even, every) non-cyclic torsion-free hyperbolic group.

\end{document}